\documentclass[A4,12pt]{article}
\usepackage{amssymb}
\usepackage{amsmath}
\usepackage{amsthm}
\usepackage{graphicx}
\usepackage{varwidth}
\usepackage{float}
\usepackage{fullpage} 
\usepackage{color}
\usepackage{nicefrac}
\usepackage{epstopdf}
\usepackage{amsfonts}
\usepackage{subfig}
\usepackage{graphicx}
\usepackage{tikz} 
\usepackage{pgfplots}
\usepackage{algorithm}
\usepackage{algpseudocode}

\DeclareMathOperator*{\argmin}{\arg\!\min}

\newtheorem{lemma}{Lemma}[section]
\newtheorem{theorem}{Theorem}[section]
\newtheorem{proposition}{Proposition}[section]

\title{On Optical Flow Models for Variational Motion Estimation}

\author{Martin Burger\thanks{Institute for Computational and Applied Mathematics and Cells in Motion Cluster of Excellence, University of M{\"u}nster, Orl\'{e}ans-Ring 10, 48149 M{\"u}nster, Germany, Email: \href{mailto:martin.burger@wwu.de}{\texttt{martin.burger@wwu.de}}}, Hendrik Dirks\thanks{Institute for Computational and Applied Mathematics and Cells in Motion Cluster of Excellence, University of M{\"u}nster, Orl\'{e}ans-Ring 10, 48149 M{\"u}nster, Germany, Email: \href{mailto:hendrik.dirks@wwu.de}{\texttt{hendrik.dirks@wwu.de}}}, Lena Frerking\thanks{Institute for Computational and Applied Mathematics and Cells in Motion Cluster of Excellence, University of M{\"u}nster, Orl\'{e}ans-Ring 10, 48149 M{\"u}nster, Germany, Email: \href{mailto:lena.frerking@wwu.de}{\texttt{lena.frerking@wwu.de}}}}

\begin{document}
	
\maketitle

\begin{abstract}
	The aim of this paper is to discuss and evaluate total variation based regularization methods for motion estimation, with particular focus on optical flow models. In addition to standard $L^2$ and $L^1$ data fidelities we give an overview of different variants of total variation regularization obtained from combination with higher order models and a unified computational optimization approach based on primal-dual methods. Moreover, we extend the models by Bregman iterations and provide an inverse problems perspective to the analysis of variational optical flow models.  
	
	A particular focus of the paper is the quantitative evaluation of motion estimation, which is a difficult and often underestimated task. We discuss several approaches for quality measures of motion estimation and apply them to compare the previously discussed regularization approaches.
\end{abstract}

\section{Introduction}
Motion estimation is a crucial task in many different areas. On the one hand, motion estimation is important in medical and biological contexts where the goal is e.g. to track moving cells or to detect the motion of organs. On the other hand, it is also important in the automotive sector. Nowadays, there are many approaches that use motion estimation to make driving saver, by detecting both dangers from the outside of a car and also inattentiveness of a driver, which expresses e.g. in slower movement or closing eyes due to tiredness. These are only two of very many further applications of motion estimation.\\
Motion estimation generally arises in the context of image sequences $u(x,t)$, depending on a spatial position $x\in\Omega \subset \mathbb{R}^d$ and a time $t\in\left[0,T\right]$. For real applications there exists only the discrete counterpart of $u$, which is a set of images recorded at time steps $t_0,t_0+\delta_t,t_0+2\delta_t,\ldots$. There also exist a variety of characteristics we have to consider when estimating motion:
\begin{itemize}
	\item A digital image can suffer from low resolution, low contrast, different gray levels and noise.
	\item The temporal resolution $\delta_t$ is strongly connected to the underlying motion. For too large time steps we might loose correspondence between consecutive images (e.g. a very fast car might only be visible in one image).
	\item A natural image often contains a set of moving objects with different speeds. A sufficient model should simultaneously be able to detect small and large movements in the same sequence. On the other hand, for the static background no motion should be detected.
	\item For many biological applications, we have to consider the fact that the illumination is constant, but fluorescence of the observed object can be inhomogeneous in space, or might even underlie changes over time.
\end{itemize}

\subsubsection{Optical Flow and Real Motion} When looking at image sequences and moving objects we directly speak of motion. This is a false implication since only projections of the real 3-dimensional motion fields are recorded by an image recording device (and in particular the human eye). \\
To emphasize this fact, we consider a camera recording traffic on a highway. The camera is only able to follow 2-dimensional paths $\boldsymbol{v}_p=(v^1,v^2)$ on the image domain $\Omega$, which is the projection of the real 3-dimensional path $\boldsymbol{v}_r=(v^1,v^2,v^3)$. Thus, already one degree of information gets lost here. This problem is even worse since we are not able to measure the 2-dimensional motion field directly. On images only the apparent motion (or \textit{optical flow}) is visible, that is displacements of intensities. Unfortunately the apparent motion and the 2-dimensional motion field are two fundamentally different properties (a detailed discussion of this problem can be found in $\cite{verri1989motion}$). The Barber's pole example is often used to underline this difference. The pole simply rotates counterclockwise, but optical devices (e.g. camera, eye) can only detect gray values tending upwards and consequently the optical flow points upwards.\\
The aim of an optical flow model is to detect the real motion only from the intensities of the given images. Performing flow estimation we have to deal with an inverse problem. We need to find the right motion that leads to the actual appearance of a specific image. 

\section{Models}

One of the most common techniques to formally link intensity variations and motion is the optical flow constraint. Based on the assumption that the image intensity $u(x,t)$ is constant along a trajectory $x(t)$ with $\frac{dx}{dt}=\boldsymbol{v}(x,t)$ we get using the chain-rule
\begin{align}
0 = \frac{du}{dt} =  \frac{\partial u}{\partial t} + \sum_{i=1}^n \frac{\partial u}{\partial x_i}\frac{dx_i}{dt} = u_t + \nabla u\cdot\boldsymbol{v}.
\label{equation:opticalFlowConstraint}
\end{align}
The last equation is generally known as the \textbf{optical flow constraint}. This optical flow constraint can also be regarded as a linear inverse problem. Therefore, we rewrite \eqref{equation:opticalFlowConstraint} by changing the positions of $\boldsymbol{v}$  and $\nabla u$ to get
\begin{align}
	(\nabla u)^T \cdot \boldsymbol{v} = -  u_t  .
\end{align}
We define $A:L^p(\Omega)^2 \rightarrow L^p(\Omega)$ via
\begin{align}
A\boldsymbol{v} = (\nabla u)^T \cdot \boldsymbol{v} &&\text{and}&& g = -u_t. \label{Adefinition}
\end{align}
Hence, the constraint fulfills the inverse problem $A\boldsymbol{v}=g$, with $p=1$ or $p=2$ depending on the norm of the data fidelity. A more detailed introduction into the optical flow problem can be found in for example in \cite{aubert2006mathematical} or \cite{becker2015optical}. We will discuss this issue in more detail in the following sections.

\subsection{Variational Models with Gradient Regularization}
\label{sec:standard}
The optical flow constraint constitutes in every point $x$ one equation, but in the context of motion estimation from images we usually have two or three spatial dimensions. Consequently, the problem is massively underdetermined. However, it is possible to estimate the motion using a variational model  
\begin{align*}
	\min_{\boldsymbol{v}} \mathcal{D}(u,\boldsymbol{v}) + \alpha \mathcal{R}(\boldsymbol{v}),
\end{align*}
where $\mathcal{D}(u,\boldsymbol{v})$ represents the so-called \textit{data term} and incorporates the optical flow constraint in a suitable norm. The second part $\mathcal{R}(\boldsymbol{v})$ models additional a-priori knowledge on $\boldsymbol{v}$ and is denoted as \textit{regularizer}. The parameter $\alpha$ regulates between data term and regularizer.\\
Possible choices for the data term are 
\begin{align*}
	\mathcal{D}_1(u,\boldsymbol{v}) := \frac{1}{2}\left\|\boldsymbol{v}\cdot\nabla u + u_t\right\|_2^2,
\end{align*}	
or
\begin{align*}
	\mathcal{D}_2(u,\boldsymbol{v}) := \left\|\boldsymbol{v}\cdot\nabla u + u_t\right\|_1.
\end{align*}	
The quadratic L$^2$ norm can be interpreted as solving the optical flow constraint in a least-squares sense inside the image domain $\Omega$. On the other hand, taking the L$^1$ norm enforces the optical flow constraint linearly and is able to handle outliers more robust. \\
The regularizer $\mathcal{R}(\boldsymbol{v})$ has to be chosen such that the a-priori knowledge is modeled in a reasonable way. If the solution is expected to be smooth, a quadratic L$^2$ norm on the gradient of $\boldsymbol{v}$ is chosen and we have
\begin{align*}
	\mathcal{R}_1(\boldsymbol{v}) := \frac{1}{2}\left\|\nabla\boldsymbol{v}\right\|_2^2.
\end{align*}
Another possible approach is to choose the \textit{total variation} (TV) of $\boldsymbol{v}$ if we expect piecewise constant parts of motion. In the finite dimensional setting this can be written as
\begin{align*}
	\mathcal{R}_2(\boldsymbol{v}) := \left\|\nabla\boldsymbol{v}\right\|_1.
\end{align*}
Taking
\begin{align*}
\mathcal{D}_1(u,\boldsymbol{v}) = \frac{1}{2}\left\|\boldsymbol{v}\cdot\nabla u + u_t\right\|_2^2 &&\text{and}&& \mathcal{R}_1(\boldsymbol{v}) = \frac{1}{2}\left\|\nabla\boldsymbol{v}\right\|_2^2
\end{align*}
results in the very well known model of Horn and Schunck \cite{horn1981determining}, 1981. With efficient primal-dual schemes to minimize L$^1$ norms \cite{chambolle2011first}, L$^1$-TV optical flow models with
\begin{align*}
\mathcal{D}_2(u,\boldsymbol{v}) = \left\|\boldsymbol{v}\cdot\nabla u + u_t\right\|_1 &&\text{and}&& \mathcal{R}_2(\boldsymbol{v}) = \left\|\nabla\boldsymbol{v}\right\|_1
\end{align*}
became very popular \cite{zach2007duality,perez2013tv}. This model was further improved by Werlberger et el \cite{werlberger2009anisotropic}, where the classical TV regularizer was replaced with a Huber norm. In this context, let us refer to a recent survey by Becker, Petra and Schn{\"o}rr \cite{becker2015optical}.

\subsection{Extension of the Regularizer}
\label{modelsExtended}
The previously described regularizations are able to produce either smooth velocity fields (L$^2$ regularization) or sharp edges (TV regularization). However, in realistic applications we often need a combination of both characteristics. To calculate velocity fields that combine smoothness and sharp edges, we need to extent the regularization term.\\
A possibility to 
In the following approach we want to make use of the potential of TV regularization to calculate sharp edges. To incorporate the possibility to get smooth parts, we need an additional primal variable $w:\mathbb{R}^2\rightarrow\mathbb{R}^2$, which will be needed for a second regularization term that allows for smooth transitions. We link the variables $\boldsymbol{v}$ and $w$ by forcing
\begin{equation*}
	\nabla \boldsymbol{v} - w \approx 0.
\end{equation*}
If we penalize this difference with an L$^1$ norm and minimize it for $\boldsymbol{v}$, we get a simple TV regularization, which is shifted by $w$. This way we ensure to keep characteristics of a TV regularization for our results. To combine this with the characteristics of another regularization technique, we apply a second regularizer to $w$ and also minimize everything for $w$. This procedure leads us to the following extended regularizer:
\begin{equation}
\label{eq:extended}
\mathcal{R}(\boldsymbol{v},w) = \alpha_0 \sum_{i=1}^d \left\| \nabla v_i - w \right\|_1 + \alpha_1 \mathcal{S} (w).
\end{equation}
The term $\mathcal{S} (w)$ is the additional term, which should be chosen in a way that it produces smooth flow fields. The weighting parameters $\alpha_0\in\mathbb{R}^+$ and $\alpha_1\in[0,1]$ on the one hand determine the relation between the regularization parts and the optical flow constraints. On the other hand they also determine the relation between the two regularization parts itself and this way adjust the proportion of smooth parts and edges in the result. If $\frac{\alpha_0}{\alpha_1}$ is large, the constant parts and edges outweigh the smooth parts; if it is small, the smooth parts outweigh the constant parts and edges. Note that a functional of the form \eqref{eq:extended} leads to a regularization of $\boldsymbol{v}$ via
\begin{equation}
 	\mathcal{\tilde R}(\boldsymbol{v} ) = \inf_w \mathcal{R}(\boldsymbol{v},w)
\end{equation}
 \\
Similar to the TV term in a standard model, we can choose between different evaluations of the norm $\left\| \nabla v_i - w \right\|_1$ in the above formulation. Since this is a shifted TV norm, we once more have the possibility to use the anisotropic variant.

\subsubsection{L$^1$-TV/L$^2$ Optical Flow Model}
As we have already seen before, a regularization with the L$^2$ norm leads to smooth approximations of a flow field. In the following approach, we use L$^2$ of $w$ as a norm for the additional regularization term $\mathcal{S} (w)$, which leads us to
\begin{equation}
\mathcal{R}(v,w) = \alpha_0 \sum_{i=1}^d \left\| \nabla v_i - w \right\|_1 +  {\frac{\alpha_1}{2}\|w\|_2^2}.
\end{equation}
Together with
\begin{align*}
\mathcal{D}(u,\boldsymbol{v}) = \left\|\boldsymbol{v}\cdot\nabla u + u_t\right\|_1
\end{align*}
we get an L$^1$-TV/L$^2$ model. This model is a combination between the previously described Horn-Schunck model and the L$^1$-TV model and thus contains characteristics of both kinds of flow fields. The regularizer of this model is very similar to the Huber function \cite{huber1964robust}, which is defined as
\begin{align*}
H_{\epsilon}(r)=\begin{cases} \frac{r^2}{2\epsilon} &0\le |r| \le \epsilon \\ |r|-\frac{\epsilon}{2} &\epsilon<|r| \end{cases}.
\end{align*}
The anisotropic variant of an optical flow model with Huber regularization has also been discussed in \cite{werlberger2009anisotropic}. For another formulation of a related regularizer see e.g. \cite{weickert2001variational} or \cite{brox2004high}.\\
Even if the Huber function can be solved directly, we want to be able to compare the different models by using the same algorithm. Therefore, later on we we will use the same the algorithm as for the other approaches.

\subsubsection{L$^1$-TV/TV Optical Flow Model}
Another possibility to consider smoothness in a flow estimation is to incorporate a higher order derivative. In the setting of \eqref{eq:extended}, we can choose $\mathcal{S}(w)$ as TV of $w$, which leads us to
\begin{equation}
\mathcal{R}(v,w) = \alpha_0 \sum_{i=1}^d \left\| \nabla v_i - w \right\|_1 + \alpha_1 {\|\nabla w\|_1}.
\end{equation}
Since the first term of the regularization leads to $w$ being close to the gradient of $w$, the last term considers a term, which is close to the second gradient of $v$. For the additional TV term in this model, we can again choose between the isotropic and the anisotropic variant, i.e. equivalently to those in Section \ref{sec:standard}. Overall, this leads to a huge variety of possible evaluations for this model.\\
By using
\begin{align*}
\mathcal{D}(u,\boldsymbol{v}) = \left\|\boldsymbol{v}\cdot\nabla u + u_t\right\|_1
\end{align*}
as a data term we get an L$^1$-TV/TV model, which we will later use for the numerical realization.\\
The regularizer we get with the described approach is a primal realization of a specific case of a Total Generalized Variation (TGV) regularizer \cite{bredies2010total}. See \cite{ranftl2012pushing,ranftl2014non} for further analysis of the TGV regularizer applied to motion estimation models.

\subsection{Bregman Iterations}
A well-known drawback of TV regularization in combination with L$^2$ data fidelity is the loss of contrast, which means that the difference between intensities in the result is lowered. However, edges are well recovered. Concerning optical flow, this shows up in a reduced velocity of motion. A possibility to overcome this drawback is to apply contrast-enhancing Bregman iterations \cite{obgxy}. For the case of higher order regularization term additional systematic bias, e.g. related to slopes, is corrected in the same way (cf. \cite{benning2013higher}). For this sake, we need the Bregman Distance of the regularizer $\mathcal{R}$, which is defined as
\begin{align*}
	D_\mathcal{R}^{\boldsymbol{b}}(\boldsymbol{v},\boldsymbol{v}^*) = \mathcal{R}(\boldsymbol{v}) - \mathcal{R}(\boldsymbol{v}^*) - \langle\boldsymbol{b},\boldsymbol{v}-\boldsymbol{v}^*\rangle, \quad \boldsymbol{b} \in \partial \mathcal{R}(\boldsymbol{v}^*),
\end{align*}
for $\boldsymbol{v},\boldsymbol{v}^* \in dom({\mathcal{R}})$. This distance replaces the former regularization term $\mathcal{R}$ in the respective models. This way it is possible to re-enhance the contrast. Thanks to the condition that $\boldsymbol{b}$ is in the subgradient of $\mathcal{R}(\boldsymbol{v}^*)$, it is ensured, that the edges stay at the correct positions. The Bregman iteration iteratively computes
$	\boldsymbol{v}^{n+1}$ as a solution of
\begin{equation}
\min_{\boldsymbol{v}} \mathcal{D}(u,\boldsymbol{v}) + \alpha D_\mathcal{R}^{\boldsymbol{b}^n}(\boldsymbol{v},\boldsymbol{v}^n).
\label{bregman}
\end{equation}
Afterwards, the Bregman variable is updated by setting 
$$\boldsymbol{b}^{n+1} = \boldsymbol{b}^n - \frac{1}\alpha (u_t + \nabla u\cdot \boldsymbol{v}^n)\nabla u.$$
Applying Bregman iterations to the optical flow problem with L$^2$ data fidelity can equivalently  adding the Bregman variable $\boldsymbol{b}^n$ to the data. For the n-th Bregman iteration we can use the extended data term $\Vert u_t + \nabla u\cdot \boldsymbol{v}\Vert^2_2 - \alpha\left\langle  \boldsymbol{b}^n,\boldsymbol{v}\right\rangle$, which results in an additive constant for the optimality condition.  So far Bregman iterations have not been investigated as an iterative regularization method for optical flow. A previous investigation \cite{hoeltgenbregman} was only considering (split) Bregman methods as a minimization method for the original variational model, not investigating the possibility to reduce systematic bias. As we shall see below (see Figure \ref{figureResults2}) the Bregman iteration can yield improvements in parts where other variational methods underestimate motion (see e.g. the right block in the rubber whale data set below) and produce competitive results (see Table \ref{parametersAndRanks}). It might thus nicely complement other approaches.
	
In the case of an $L^1$ data fidelity, the simple modification of the data term is not possible and the iteration needs to be carried out via the original form (\ref{bregman}). Due to the additional linear term $\langle \boldsymbol{b}^n, \boldsymbol{v}\rangle$ in comparison to the other only linear terms in the functional, it is not guaranteed that the functional is coercive, hence the well-posedness of the iteration is not guaranteed in the infinite-dimensional setting. In any case, it is not sure that the Bregman iteration yields additional benefit, since the L$^1$ data fidelity does not introduce systematic bias.  We leave this issue to future research.

\section{Analysis}

In the following we briefly comment on some aspects in the analysis of variational problems for optical flow, with special focus on particular aspects related to the recent extensions of total variation. We are able to generalize existing results for the Horn-Schunck and $L^2$-TV to other $L^p$ data terms, focusing in particular on $p=1$. This can be done under the same conditions on the image gradient as previously used, but the proof needs some adaption to take care of the constant functions in the null-space of the gradient. The coercivity related to the latter needs a different proof compared to previous results (cf. \cite{schnoerr}), while on the other hand we give a more concise proof as in \cite{hinterberger2002analysis} and avoid too strong regularity conditions. The understanding of the basic structure of the proof then immediately allows a generalization to existence results for the different recent variants of total variation regularization, whose analysis has not been studied in the optical flow setting. Finally, we will discuss another problem hardly studied in the optical flow literature previously, namely the quantitative estimation of errors in the motion estimate when changing image data.

\subsection{Existence of Minimizers}

The basic issue of existence of minimizers for optical flow problems has been considered by various authors (cf. e.g. \cite{aubert2006mathematical,hinterberger2002analysis}) using methods of calculus of variations. With the assumption $u_t \in L^2(\Omega)$ and $\nabla u \in L^\infty(\Omega)$ it is straightforward to see that the operator $A$ and data $g$ defined in (\ref{Adefinition}) are bounded for $p=1,2$ and with standard convexity arguments lower semicontinuity of the data fidelities follow (for $p=1$ assumptions can even be relaxed). The associated data terms can all be shown to be L$^2$-coercive on subspaces of $H^1(\Omega)^2$ in the case of the Horn-Schunck model and of $BV(\Omega)$ for the other models. The subspaces are given by functions with mean value zero for standard models, while they are given by functions of vanishing mean and first moments in the case of the TV-TV and TV-L$^2$ regularizer. Hence, in order to obtain coercivity and thus existence of minimizers it is a natural question whether the data term yields coercivity on the whole space, i.e. boundedness of the data fidelity further implies bounds on the mean (and potentially the first moments) of $\boldsymbol{v}$. Obviously this cannot be true without further assumptions on $\nabla u$ (which is seen immediately for constant $u$). In the case of the Horn-Schunck model this question was answered in \cite{schnoerr}, who showed coercivity for the original Horn-Schunck model (see also \cite{lorenzURL} for a detailed discussion). His proof heavily uses the quadratic structure of the data fidelity as well as the fact that functions in the kernel of the regularization term are constant, hence it can be generalized to the L$^2$-TV model, but neither to extensions of the TV regularization nor to L$^1$-fidelities. However, the result can easily be generalized as we show in the following:

\begin{lemma} \label{lemma1}
Let $\Omega \subset \mathbb{R}^d$, $d \geq 1$ arbitrary, be a bounded domain and let $(\boldsymbol{v}^n)$ be a sequence in $L^p(\Omega)$, $1 \leq p \leq \infty$. Let $u_t \in L^p(\Omega)$ and $\nabla u \in L^\infty(\Omega)^2$ such that the functions $\partial_{x_i} u$ are linearly independent. If there exists a sequence of vectors $c^n \in \mathbb{R}^d$ such that $c^n + \boldsymbol{v}^n$ is bounded in $L^p(\Omega)$ and 
$$ \Vert u_t + \nabla u \cdot \boldsymbol{v}^n \Vert_{L^p}\text{ is bounded in } \mathbb{R}^+,$$
then $\boldsymbol{v}^n$ is bounded in $L^p(\Omega)$
\end{lemma}
\begin{proof} 
Under the above assumptions it obviously suffices to prove that the sequence $c^n$ is bounded in $\mathbb{R}^d$.We have 
\begin{align*}
 \Vert u_t + \nabla u \cdot \boldsymbol{v}^n \Vert_{L^p} &= \Vert u_t + \nabla u \cdot(c^n + \boldsymbol{v}^n) -
\nabla u \cdot c^n\Vert_{L^p} \\ &\geq \Vert \nabla u \cdot c^n\Vert_{L^p}- \Vert u_t + \nabla u \cdot(c^n + \boldsymbol{v}^n  )\Vert_{L^p}.
\end{align*}
The boundedness of $ \Vert u_t + \nabla u \cdot \boldsymbol{v}^n \Vert_{L^p} $ of $c^n + \boldsymbol{v}^n$ imply that 
$\nabla u \cdot c^n$ is bounded in $L^p(\Omega)$. Now assume that $c^n$ is unbounded, then there exists a subsequence (again denoted by $c^n$ with $|c^n|$ growing to infinity in a monotone way. Then $\frac{c^n}{|c^n|}$ is well-defined and bounded, moreover $\nabla u \cdot \frac{c^n}{|c^n|}$ tends to zero in $L^p(\Omega)$. Hence,  $\frac{c^n}{|c^n|}$ has a convergent subsequence $d^n$ with limit $d$ satisfying $|d|=1$, in particular $d \neq 0$. By continuity of the norm we find
$$ \Vert \nabla u \cdot d \Vert_{L^p} = 0, $$
which contradicts the linear independence of the functions $\partial_{x_i} u$ and thus yields the assertion.
\end{proof}

With the Poincare-Wirtinger inequality and the embedding of $BV(\Omega)$ into $L^2(\Omega)$ for $d \leq 2$ and $L^1(\Omega)$ in arbitrary dimension, we obtain the coercivity of the total variation on functions of mean value zero. Now we choose $c^n$ equal to the mean value of $\boldsymbol{v^n}$ and obtain coercivity of the respective L$^p$-TV model in $BV(\Omega)$. With a standard (weak) lower semicontinuity argument for the convex functionals we thus deduce the following existence result:
\begin{theorem}
Let $\Omega \subset \mathbb{R}^d$ be a bounded domain and let either $p=1$ or $p=2$. Let $u_t \in L^p(\Omega)$ and $\nabla u \in L^\infty(\Omega)^d$ such that the functions $\partial_{x_i} u$ are linearly independent. Then there exists a minimizer $\boldsymbol{v} \in BV(\Omega)^d$ for the L$^p$-TV model in the cases $p=1$ for arbitrary $d$ and $p=2$ for $d \leq 2$.
\end{theorem}

\begin{lemma}  
Let $\Omega \subset \mathbb{R}^d$, $d \geq 1$ arbitrary, be a bounded domain and let $(\boldsymbol{v}^n)$ be a sequence in $L^p(\Omega)$, $1 \leq p \leq \infty$. Let $u_t \in L^p(\Omega)$ and $\nabla u \in L^\infty(\Omega)^2$ such that the functions $\{\partial_{x_i} u,x_j \partial_{x_i} u\}_{i,j=1,\ldots,d}$ are linearly independent. If there exists a sequence of vectors $c^n \in \mathbb{R}^d$ and matrices $A^n \in \mathbb{R}^{d \times d}$such that $c^n + A^n x + \boldsymbol{v}^n$ is bounded in $L^p(\Omega)$ and 
$$ \Vert u_t + \nabla u \cdot \boldsymbol{v}^n \Vert_{L^p}\text{ is bounded in } \mathbb{R}^+,$$
then $\boldsymbol{v}^n$ is bounded in $L^p(\Omega)$
\end{lemma}
\begin{proof} 
Under the above assumptions it obviously suffices to prove that the sequence $c^n$ is bounded in $\mathbb{R}^d$ and the
sequence $A^n$ is bounded in $\mathbb{R}^{d \times d}$. We have 
\begin{align*}
 \Vert u_t + \nabla u \cdot \boldsymbol{v}^n \Vert_{L^p} &= \Vert u_t + \nabla u \cdot(c^n + A^n x+ \boldsymbol{v}^n) -
\nabla u \cdot c^n\Vert_{L^p} \\ &\geq \Vert \nabla u \cdot c^n\Vert_{L^p}- \Vert u_t + \nabla u \cdot(c^n +  A^n x+ \boldsymbol{v}^n)  \Vert_{L^p}.
\end{align*}
The boundedness of $ \Vert u_t + \nabla u \cdot \boldsymbol{v}^n \Vert_{L^p} $ of $c^n + A^n x + \boldsymbol{v}^n$ imply that 
$\nabla u \cdot (c^n+A^nx)$ is bounded in $L^p(\Omega)$. A limiting procedure as in the proof of Lemma \ref{lemma1} yields a nontrivial limit $d$ respectively $B$ such that
$$ \Vert \nabla u \cdot (d+BX) \Vert_{L^p} = 0, $$
which contradicts the linear independence assumption and thus yields the assertion.
\end{proof}

Using the coercivity of the effective functionals $\mathcal{\tilde  R}$ on the space of BV-functions with vanishing mean and first moments (cf. \cite{benning2013higher} in the TV/TV case and \cite{burger2015infimal} in the TV/L$^2$ case) and a similar proof as above we obtain the following result:
\begin{theorem}
Let $\Omega \subset \mathbb{R}^d$ be a bounded domain and let either $p=1$ or $p=2$. Let $u_t \in L^p(\Omega)$ and $\nabla u \in L^\infty(\Omega)^d$ such that the functions $\{\partial_{x_i} u,x_j \partial_{x_i} u\}_{i,j=1,\ldots,d}$ are linearly independent. Then there exists a minimizer $\boldsymbol{v} \in BV(\Omega)^d$ for the L$^p$-TV/TV model and the L$^p$-TV/L$^2$ model in the cases $p=1$ for arbitrary $d$ and $p=2$ for $d \leq 2$.
\end{theorem}

A major restriction of the existing existence analysis is the fact that some regularity of $u_t$ and $\nabla u$ needs to be assumed, which might not be met by solutions of transport equations, in particular for noisy versions of the images (note however that many results in the literature assume even stronger regularity, e.g. images of class $C^2$ as in \cite{hinterberger2002analysis}). A detailed analysis of such issues remains a relevant question for future work. The case of stochastic noise in the image data might even enforce different approaches, e.g. Bayesian models (cf. \cite{suhr2014registration}).

Let us finally mention that uniqueness of minimizers cannot be expected for total variation models, since neither the regularization nor the data term are strictly convex (due to the nullspace of the map $\boldsymbol{v} \mapsto \nabla u \cdot \boldsymbol{v}$). In the case of an L$^2$ data fidelity one can at least derive uniqueness of the component $\nabla u \cdot \boldsymbol{v}$ normal to the image level sets.

\subsection{Quantitative Estimates}

Quantitative estimates have hardly been investigated for optical flow models in the past, which is probably due to the nonuniqueness of the flow reconstruction. Hence, an estimate to a ground truth seems out of reach in most cases. Nonetheless, it seems interesting to take a look at the optical flow problem from the perspective of error estimation for inverse problems (cf. \cite{burger2004convergence}). Let us consider the case of L$^2$ data fidelities for simplicity, we start from the optimality condition
\begin{equation}
	(u_t + \nabla u \cdot \boldsymbol{v}) \nabla u + \alpha \boldsymbol{b} = 0, \qquad \partial \boldsymbol{b} \in \partial \mathcal{R}(\boldsymbol{v}).
\end{equation}
The obvious  first question to ask is the robustness of solutions for errors in the data $u$, a second question is related to the behaviour as $\alpha$ tends to zero. Let us first sketch the derivation of quantitative error estimates in the case of data perturbations, for this sake let $\boldsymbol{\tilde v}$ be the solution of the variational model with data $\tilde u_t$ and $\nabla \tilde u$. Then we have
$$
	 \nabla u \cdot(\boldsymbol{v}-\boldsymbol{\tilde v})  \nabla u + \alpha ( \boldsymbol{b} - \boldsymbol{ \tilde b}) =  \tilde u_t \nabla \tilde u- u_t
	\nabla u+ \nabla \tilde u \cdot \boldsymbol{\tilde v}  \nabla \tilde u - \nabla u \cdot \boldsymbol{\tilde v}  \nabla u .
$$
A duality product with $\boldsymbol{v}-\boldsymbol{\tilde v}$ and an elementary calculation yields
\begin{align*}&  \Vert \nabla u \cdot(\boldsymbol{v}-\boldsymbol{\tilde v}) \Vert_{L^2}^2  + \alpha D_\mathcal{R}^{\boldsymbol{b}}(\boldsymbol{\tilde v},\boldsymbol{v}) + \alpha D_\mathcal{R}^{\boldsymbol{\tilde b}}(\boldsymbol{v},\boldsymbol{\tilde v}) = \\
& \qquad \qquad \qquad \langle \nabla \tilde u \tilde u_t - \nabla u u_t
	+ \nabla \tilde u \cdot \boldsymbol{\tilde v}  \nabla \tilde u - \nabla u \cdot \boldsymbol{\tilde v}  \nabla u, \boldsymbol{v}-\boldsymbol{\tilde v} \rangle
	\end{align*}
Using appropriate duality estimates on the right-hand side one immediately obtains estimates for the Bregman distances and the error in the velocity component parallel to $\nabla u$ in terms of errors in the data. To do so, let us further rewrite the right-hand side to obtain
\begin{align*}&  \Vert \nabla u \cdot(\boldsymbol{v}-\boldsymbol{\tilde v}) \Vert_{L^2}^2  + \alpha D_\mathcal{R}^{\boldsymbol{b}}(\boldsymbol{\tilde v},\boldsymbol{v}) + \alpha D_\mathcal{R}^{\boldsymbol{\tilde b}}(\boldsymbol{v},\boldsymbol{\tilde v}) = \\
& \qquad \qquad \qquad \langle   \tilde u_t +   \nabla \tilde u \cdot \boldsymbol{\tilde v}-  u_t - \nabla u \cdot \boldsymbol{\tilde v}, \nabla u \cdot (\boldsymbol{v}-\boldsymbol{\tilde v}) \rangle +\\
&	 \qquad \qquad \qquad \langle (\nabla \tilde u-\nabla u)   (\tilde u_t + \nabla \tilde u \cdot \boldsymbol{\tilde v}), \boldsymbol{v}-\boldsymbol{\tilde v} \rangle
	\end{align*}
and with Young's inequality we find
\begin{align*}
  \frac{1}{2}\Vert \nabla u \cdot(\boldsymbol{v}-\boldsymbol{\tilde v}) \Vert_{L^2}^2  + \alpha D_\mathcal{R}^{\boldsymbol{b}}(\boldsymbol{\tilde v},\boldsymbol{v}) + \alpha D_\mathcal{R}^{\boldsymbol{\tilde b}}(\boldsymbol{v},\boldsymbol{\tilde v})   \leq &\frac{1}2 \Vert (\tilde u_t-u_t) + (  \nabla \tilde u -\nabla u)\cdot \boldsymbol{\tilde v}\Vert_{L^2}^2 + \\&
\langle (\nabla \tilde u-\nabla u)   (\tilde u_t + \nabla \tilde u \cdot \boldsymbol{\tilde v}), \boldsymbol{v}-\boldsymbol{\tilde v} \rangle
	\end{align*}
The two error terms on the right-hand side can be interpreted well. The first term measures the difference of the image data in a norm related to the optical flow model, it can further be related to the consistency in the optical flow constraint with estimated velocity ${\boldsymbol{\tilde v}}$. The second term creates more trouble, since due to $\nabla \tilde u$ it can contain terms in the direction perpendicular to $\nabla u$, in which case it cannot be related to the first term on the left-hand side. Note that when viewing the optical flow estimation as an inverse problem this term is indeed related to the model error, i.e. the error in the forward operator
$\boldsymbol{v} \mapsto \boldsymbol{v} \cdot \nabla u$. Even if this term can be estimated somehow by the Bregman distances on the left-hand side (e.g. in the case of the Horn-Schunck functional simply by Young's inequality), it will lead to an error term of order $\frac{1}\alpha$. With the inherent smallness of $\alpha$ we hence conclude that errors in the image gradient will strongly influence the result. This can e.g. be interpreted in terms of the accuracy in the discrete approximation of the image gradient. The quality in the estimated flow indeed depends heavily on the gradient approximation in numerical results. This is illustrated in Table 2 via an evaluation of errors for different choices of difference quotients for $\nabla u$.

In the case of negligible error in the image gradient we can give an improved estimate:
\begin{proposition}
Let $u$ and $\tilde u$ be weakly differentiable image data such that $\partial_t u, \partial_t \tilde u \in L^2(\Omega)$ and $\nabla u = \nabla \tilde u$. Then the estimate 
\begin{equation}
	\frac{1}{2}\Vert \nabla u \cdot(\boldsymbol{v}-\boldsymbol{\tilde v}) \Vert_{L^2}^2  + \alpha D_\mathcal{R}^{\boldsymbol{b}}(\boldsymbol{\tilde v},\boldsymbol{v}) + \alpha D_\mathcal{R}^{\boldsymbol{\tilde b}}(\boldsymbol{v},\boldsymbol{\tilde v})   \leq \frac{1}2 \Vert \tilde u_t -  u_t  \Vert_{L^2}^2
\end{equation}
holds for $v$ and $\tilde v$ being the solution of  
\begin{align*}
	\min_{\boldsymbol{v}} \frac{1}{2}\left\|\boldsymbol{v}\cdot\nabla u + u_t\right\|_{L^2}^2  + \alpha \mathcal{R}(\boldsymbol{v}),
\end{align*}
with data $u$ respectively $\tilde u$.
\end{proposition}

Note that we see in any estimate that the velocity parallel to the image gradient (normal to level lines) can be estimated robustly from the image error (with constants independent of $\alpha$), while we obtain hardly information perpendicular to the gradient. This is just a mathematical consequence of the aperture problem in motion estimation. In the perpendicular direction we obtain some stability of the velocity via the Bregman distance terms, but note that the error will grow inversely proportional to $\alpha$. We mention that to leading order in a small time step $\tau$ we have
$$ \Vert \nabla u \cdot(\boldsymbol{v}-\boldsymbol{\tilde v}) \Vert_{L^2} \approx
 \Vert u (\cdot + \tau  \boldsymbol{v}(\cdot))-u (\cdot + \tau \boldsymbol{\tilde v} (\cdot)) \Vert_{L^2}.$$ 
Hence, the error measure is strongly related to the end-point error frequently used in practical evaluations of motion estimation, an issue we will discuss in detail below.
 
The limit $\alpha \rightarrow 0$ is the one usually studied in inverse problems with strong emphasis on the appropriate treatment of instabilities. In the case of motion estimation the underdeterminedness is the more crucial issue. Thus, the key question is to understand what solutions are obtained in the limit, respectively which kind of flows can be reconstructed well. The limiting solutions are determined by the constrained problem
\begin{equation}
 \min_{\boldsymbol{v}}	\mathcal{R}(\boldsymbol{v}) \quad \text{subject to } u_t + \nabla u \cdot \boldsymbol{v} = 0.
\end{equation}
The solutions to be reconstructed well are those satisfying a {\em source condition} (cf. \cite{burger2004convergence}), i.e. the optimality condition for the constrained problem in a Lagrangian duality setting. This is equivalent to
\begin{equation}
	\boldsymbol{b} = w \nabla u \in \partial \mathcal{R}(\boldsymbol{v}) \label{sourcecondition}
\end{equation}
for some $w \in L^2(\Omega)$. Under this condition one can construct data $\tilde u_t = \tilde u_t + \alpha w$, such that the constrained solution also satisfies the unconstrained variational problem with regularization parameter $\alpha >0$, thus error estimation reduces to the above case. From \eqref{sourcecondition} we can give a mathematical reason why it is  it is useful to have images with a large gradient, which corresponds to the usual intuition. The data perturbation is a multiple of $w$, hence error estimates depend on the norm of $w$, which is however inversely proportional to the norm of $\nabla u$. A more detailed study of  \eqref{sourcecondition} depends on the specific regularization terms and is an interesting question for further research.

\section{Numerical Solution}

In the following we discuss the numerical solution of the variational models discussed above. We start with a unified discussion of primal-dual iterative methods and then proceed to discretization issues. 

\subsection{Primal-Dual Algorithm}
The concept of duality offers an efficient way of minimizing the introduced variational models. Let us for the numerical application consider two finite dimensional vector spaces $\mathcal{X}$ and $\mathcal{Y}$ equipped with a scalar product $\langle\cdot,\cdot\rangle$ and a norm $\left\|\cdot\right\|_2$. We furthermore consider a continuous linear operator $K:\mathcal{X}\rightarrow\mathcal{Y}$. Now, class of variational problems in this section can be written as 
\begin{align}
	\min_{x\in\mathcal{X}} G(v) + F(Kv),
	\label{cpEqPrimal}
\end{align}
with $F,G : \mathcal{X}\rightarrow\mathbb{R}$ are proper, convex and lower semi-continuous functionals. In our application, $G(v)$ usually incorporates the data term, whereas $ F(Kv)$ denotes the regularizer. In the following, we denote Equation $\eqref{cpEqPrimal}$ as the \textit{primal problem}. The primal problem is often hard to minimize or yields very slow algorithms. Instead of the primal problem, we can equivalently optimize the so-called \textit{primal-dual problem}:
\begin{align}
	\min_{x\in\mathcal{X}}\max_{y\in\mathcal{Y}} \left\langle Kv,y \right\rangle + G(v) - F^*(y).
	\label{cpEqSaddlepoint}
\end{align}
Here $F^*$ refers to the convex conjugate of $F$, which can be easily calculated for the observed regularizers. The very well known algorithm of Chambolle and Pock dedicates to this class of problems \cite{chambolle2011first}. For $\tau,\sigma>0$, a pair $(v^0,y^0)\in\mathcal{X}\times\mathcal{Y}$ and initial value $\bar{v}^0=0$ we obtain the following iterative scheme to solve the saddle-point problem  $\eqref{cpEqSaddlepoint}$:
\begin{align}
	y^{k+1} &= prox_{\sigma F^*}(y^k+\sigma K\hat{x}^k)\\
	\label{eq:iterativex}
	v^{k+1} &=  prox_{\tau G}(v^k-\tau K^* y^{k+1})\\
	\hat{v}^{k+1} & = 2 v^{k+1} -v^k
\end{align}
Here, $prox_{\tau G}$ denotes the \textit{proximal} or \textit{resolvent} operator
\begin{align*}
	prox_{\tau G}(y) = \left( I+\tau \partial G\right)^{-1}(y) := \argmin_v \left\{ \frac{\left\| v-y \right\|_2^2}{2}+\tau G(v) \right\} ,
\end{align*}
which can be interpreted as a compromise between minimizing $G$ and being close to the input argument $y$. Suitable choices of $\tau$ and $\sigma$ were chosen according to \cite{pock2011diagonal} and are given in Section \ref{sectionDiscretization} for each algorithm. The primal-dual methods are efficient if the proximal problem can be solved efficiently. In \cite{chambolle2011first} the authors demonstrated an application of the primal-dual algorithm to a L$^1$-TV optical flow problem. In the following we will extend this application to each of the introduced models and provide solutions to the occurring proximal problems.\\
In Table \ref{tab:listPrimalDualNotation} we transferred our models to the notation of the primal problem \eqref{cpEqPrimal}. Let us begin with the proximal problems for the primal part $G(v)$, for a more detailed description we refer to \cite{dirks}.
\paragraph{L$^2$ data term:}
For the L$^2$ data term we have to solve for $\tilde{v}^{k+1} := v^k-\tau K^* y^{k+1}$ the following proximal problem:
\begin{align*}
	\argmin_v \left\{ \frac{\left\| v-\tilde{v}^{k+1} \right\|_2^2}{2}+\frac{\tau}{2} \left\|v\cdot\nabla u + u_t\right\|_2^2 \right\}.
\end{align*} 
This yields a linear optimality system, which can be directly calculated. The solution is given by $v = f(\tilde{v}^k)$, where $f$ is a linear function also incorporating the scalar terms $u_t,\nabla u$ and $\tau$.
\paragraph{L$^1$ data term:}
In case of the L$^1$ optical flow term, the proximal problem reads
\begin{align*}
	\argmin_v \left\{ \frac{\left\| v-\tilde{v}^{k+1} \right\|_2^2}{2}+ \tau \left\|v\cdot\nabla u + u_t\right\|_1 \right\}.
\end{align*}
Defining the affine linear function $\rho(v) := v\cdot\nabla u + u_t$, the problem above represents a affine-linear L$^1$ optimization problem. It can be shown that the solution is given by
\begin{align*}
	v = \tilde{v}^{k+1} + \begin{cases}
		\tau \nabla u &\mbox{if } \rho(\tilde{v}^{k+1}) < -\tau \left\|\nabla u\right\|_2^2 \\  
		-\tau \nabla u &\mbox{if } \rho(\tilde{v}^{k+1}) > \tau \left\|\nabla u\right\|_2^2 \\
		-\frac{\rho(\tilde{v}^{k+1})}{\left\|\nabla u\right\|_2^2} \nabla u &\mbox{else }
	\end{cases}.
\end{align*}
\begin{table}
\footnotesize
	\centering
	\begin{tabular}{|c|c|c|c|}
		\hline
			\textbf{Model} & \textbf{G(v)} & \textbf{F(Kv)} & \textbf{K}\\
		\hline
			L$^2$-L$^2$ &  $\frac{1}{2}\left\|v\cdot\nabla u + u_t\right\|_2^2$ & $\frac{\alpha}{2}\left\|\nabla v\right\|_2^2$ &  $\begin{pmatrix}\nabla&0\\0&\nabla\end{pmatrix}$\\
		\hline
			L$^2$-TV &  $\frac{1}{2}\left\|v\cdot\nabla u + u_t\right\|_2^2$ & $\alpha\left\|\nabla v\right\|_1$ &  $\begin{pmatrix}\nabla&0\\0&\nabla\end{pmatrix}$\\
		\hline
			L$^1$-TV &  $\left\|v\cdot\nabla u + u_t\right\|_1$ & $\alpha\left\|\nabla v\right\|_1$ &  $\begin{pmatrix}\nabla&0\\0&\nabla\end{pmatrix}$\\
		\hline\hline
		\textbf{Model} & \textbf{G(v,w)} & \textbf{F(K(v,w))} & \textbf{K}\\
		\hline
			L$^1$-TV/L$^2$ &  $\left\|v\cdot\nabla u + u_t\right\|_1 + \frac{\alpha_1}{2}\left\|w\right\|_2^2$ & $\alpha_0\left\|\nabla v-w\right\|_1$ &  $\begin{pmatrix}\nabla & 0 & -I & 0\\0&\nabla & 0 & -I\end{pmatrix}$\\
		\hline
			L$^1$-TV/TV &  $\left\|v\cdot\nabla u + u_t\right\|_1$ & $\alpha_0\left\|\nabla v-w\right\|_1 + \alpha_1\left\|\nabla w\right\|_1$ &  $\begin{pmatrix}\nabla & 0 & -I & 0\\0&\nabla & 0 & -I\\0 & 0 & \nabla & 0\\0 & 0 & 0 & \nabla\end{pmatrix}$\\
		\hline
	\end{tabular}
	\label{tab:listPrimalDualNotation}
\end{table}

\paragraph{L$^2$ regularization:}
Denoting the proximal problems for the dual part requires calculating the convex conjugate $F^*$ of the dual part $F$ first. For the L$^2$ regularization term $F(Kv)=\frac{\alpha}{2}\left\|\nabla v\right\|_2^2$ we calculate $F^*(y) = \frac{1}{2\alpha}\left\|y\right\|_2^2$ and obtain the proximal problem
\begin{align*}
	\argmin_v \left\{ \frac{\left\| y-\tilde{y}^{k+1} \right\|_2^2}{2}+ \tau \frac{1}{2\alpha}\left\|y\right\|_2^2 \right\}.
\end{align*}
Consequently, the proximal problem can be solved straightaway. 
\paragraph{TV regularization:}
In case of the total variation and the higher order regularization we always have L$^1$ norms involved. Similar to the standard TV regularization ($F(Kv)=\alpha\left\|\nabla v\right\|_1$) we have $F^*(y) = \alpha\delta_{B(L^\infty)}(y/\alpha)$, where $\delta_{B(L^\infty)}$ denotes the indicator function of the L$^\infty$ unit ball. The corresponding proximal problem then reads
\begin{align*}
	\argmin_v \left\{ \frac{\left\| y-\tilde{y}^{k+1} \right\|_2^2}{2}+ \tau \alpha\delta_{B(L^\infty)}(y/\alpha) \right\}.
\end{align*}
The solution is given by a point-wise projection of $\tilde{y}^{k+1}$ onto the unit ball corresponding to the chosen vector norm. The solution is will be discussed in more detail in Section \ref{sectionDiscretization}.

\subsection{Discretization and Parameters}
\label{sectionDiscretization}
In the context of discretization we have to consider two different aspects. On the one hand we have to approximate the image derivatives $u_t,u_x$ and $u_y$, which are passed to the algorithm as static variables afterwards. One the other hand, we have to discretize the operators $K$ and $K^*$ coming from the primal-dual iteration.\\
Starting with the image discretization, we consider the image domain to consist of the regular grid:
\begin{align*} 
	\left\lbrace (t,i,j) :t=0,1,\quad i=0,\ldots,n_x,\quad j=0,\ldots,n_y \right\rbrace
\end{align*}
Since we are interested in the velocity field between $t=0$ and $t=1$ we have to evaluate the derivatives of $u^{t,i,j}$ at $t=0$. In our observation, a mixed scheme consisting of a forward difference for the time derivative and central differences for the spatial derivatives turned out to give the best results. Stability in terms of $u$ is not required, due to the fact that they act as constants in the model. Consequently, we obtain the following scheme:
\begin{align*}
	u_t^{i,j} &= u^{1,i,j}-u^{0,i,j}\\ 
	u_x^{i,j} &= \left.\begin{cases} u^{0,i+1,j}-u^{0,i-1,j} &\mbox{if } i>0 \mbox{ and } i<n_x \\ 
		0 & \mbox{if } i=0 \mbox{ or } i=n_x \end{cases}\right\}\\
	u_y^{i,j} &= \left.\begin{cases} u^{0,i,j+1}-u^{0,i,j-1} &\mbox{if } j>0 \mbox{ and } j<n_y \\ 
		0 & \mbox{if } j=0 \mbox{ or } j=n_y \end{cases}\right\}
\end{align*}
For the sake of simplicity, we omitted the time-dependency in the derivatives. \\
The operators in our model only consist of gradients for the spatial regularization and the corresponding divergences. We propose a forward discretization with Neumann boundary conditions for the gradient and consequently obtain a backward difference for the divergence to keep a discrete adjoint structure. The resulting scheme reads
\begin{align*}
	v_x^{i,j} &= \left.\begin{cases} v^{i+1,j}-v^{i,j} &\mbox{if } i<n_x \\ 
		0 & \mbox{if } i=n_x \end{cases}\right\}\\
	v_y^{i,j} &= \left.\begin{cases} v^{i,j+1}-v^{i,j} &\mbox{if } j<n_y \\ 
		0 & \mbox{if } j=n_y \end{cases}\right\}\\
	\nabla\cdot \boldsymbol{y}^{i,j} &= 
	\left.\begin{cases} 
		y_1^{i,j}-y_1^{i-1,j} &\mbox{if } i>0 \\ 
		y_1^{i,j} & \mbox{if } i=0\\
		-y_1^{i-1,j} & \mbox{if } i=n_x 
	\end{cases}\right\} +\left.\begin{cases} 
		y_2^{i,j}-y_2^{i,j-1} &\mbox{if } j>0 \\ 
		y_2^{i,j} & \mbox{if } j=0 \\
		-y_2^{i,j-1} & \mbox{if } j=n_y
	\end{cases}\right\}
\end{align*} 
The primal-dual algorithms moreover requires suitable parameters $\sigma$ and $\tau$ such that $\sigma\tau \leq \left\| K \right\|$. It has been shown in \cite{pock2011diagonal} that a diagonal preconditioning technique can be used to find $\tau,\sigma$ without calculating $\left\|K\right\|$. According to \cite{pock2011diagonal}, Lemma 2, we choose $\tau=\frac{1}{4},\sigma=\frac{1}{2}$ for the gradient models from Section \ref{sec:standard} and $\tau=\frac{1}{5},\sigma=\frac{1}{3}$ for the extended models from Section \ref{modelsExtended}.\\
In the previous section, the solution of a proximal problem involving an indicator function of the L$^\infty$ unit ball of $y$ was given as a point-wise projection. To conceive this, let us begin by writing down the convex set corresponding to $\delta_{B(L^\infty)}(y/\alpha)$:
\begin{align*}
	\left\lbrace y : \left\| y / \alpha\right\|_\infty \leq 1 \right\rbrace.
\end{align*}
Here, $\left\| y / \alpha\right\|_\infty$ refers to the discrete maximum over all elements $(i,j)$ in the image domain:
\begin{align*}
	\left\|y/\alpha\right\|_\infty = \max_{(i,j)}  \left| y_{i,j}/\alpha \right|.
\end{align*}
An interesting aspect is how to evaluate the vector norm $\left|y_{i,j}\right| = \left|(y^1,y^2,y^3,y^4)_{i,j}\right|$. Choosing
\begin{align*}
	\left|y_{i,j}\right| := \left|y^1_{i,j}\right| + \left|y^2_{i,j}\right| + \left|y^3_{i,j}\right| + \left|y^4_{i,j}\right|
\end{align*}
results in the so-called \textit{anisotropic} variant, whereas 
\begin{align*}
	\left|\nabla\boldsymbol{v}_{i,j}\right| := \sqrt{ (y^1_{i,j})^2 + (y^2_{i,j})^2 + (y^3_{i,j})^2 + (y^4_{i,j})^2}
\end{align*}
is denoted as the \textit{fully isotropic} one. The anisotropic total variation is more suitable when dealing with quadratic structures, the isotropic variant better fits to circular structures. Depending on the chosen vector norm for $v$ (anisotropic, fully isotropic), we obtain a different dual inner norm $\left| y_{i,j}/\alpha \right|$. For the anisotropic case, that refers to the vectorial L$^1$ norm, we get the maximum norm for the dual variable:
\begin{align*}
	\left| y_{i,j} \right| = \max\left\lbrace\left| y^{1,1}_{i,j} \right|,\left| y^{1,2}_{i,j}\right|,\left|y^{2,1}_{i,j} \right|,\left| y^{2,2}_{i,j} \right|\right\rbrace.
\end{align*}
Since the fully isotropic case refers to the self-dual Euclidean norm, we get in this case
\begin{align*}
	\left| y_{i,j} \right| = \left\| y_{i,j} \right\|_2 = \sqrt{ (y^{1,1}_{i,j})^2 + (y^{1,2}_{i,j})^2 + (y^{2,1}_{i,j})^2 + (y^{2,2}_{i,j})^2 }.
\end{align*}
Transferring this to the optimization problem above, we get in the anisotropic case a the point-wise projection of $\tilde{y}^{k+1}$ onto $[-\alpha,\alpha]$:
\begin{align*}
	y^{k+1} = \min(\alpha,\max(-\alpha,\tilde{y}^{k+1})),
\end{align*}
and for the isotropic case a point-wise projection of $\tilde{y}^{k+1}$ onto the L$^2$ unit ball:
\begin{align*}
	y^{k+1}_{i,j} = \frac{\tilde{y}^{k+1}_{i,j}}{\max(1,\left\| \tilde{y}^{k+1}_{i,j} / \alpha \right\|_2)}.
\end{align*}
For isotropic and anisotropic variants of the extended models, the norms are achieved equivalently. However, it has to be considered that for the L$^1$-TV/TV model there are independent norms for both parts of the regularization. Using the partially isotropic variant also results in independent norms for the different components of $v$. For the sake of simplicity we will restrict ourselves to the fully isotropic cases for all appearing total variations in the numerical realizations.

\section{Results}

\subsection{Error Measures for Velocity Fields}

Finding error measures for velocity fields is a delicate problem which can be motivated by the following simple example. Consider two images, each consisting of only $3\times 3$ pixels of the following form:
\begin{align}\label{align:einfachesBeispielBewegung}
	u_1 = \begin{pmatrix}1&1&0\\0&0&0\\0&0&0\end{pmatrix},\quad u_2 = \begin{pmatrix}0&0&0\\1&1&0\\0&0&0\end{pmatrix}.
\end{align}
From our point of view we seek for a velocity field $\boldsymbol{v}= (v^1,v^2)$ that transforms $u_1$ into $u_2$. Unfortunately, it is unclear which velocity field underlies this motion. Examples of possible solutions are
\begin{enumerate}
	\item Just move the pixels at position $(1,1)$ and $(1,2)$ by 1 to the bottom, hence 
	\begin{align*}
		v^1 = \begin{pmatrix}1&1&0\\0&0&0\\0&0&0\end{pmatrix},\quad v^2 = \begin{pmatrix}0&0&0\\0&0&0\\0&0&0\end{pmatrix}.
	\end{align*}
	\item Move the pixels at position $(1,1)$ and $(1,2)$ by 1 to the bottom and change their position:
	\begin{align*}
		v^1 = \begin{pmatrix}1&1&0\\0&0&0\\0&0&0\end{pmatrix},\quad v^2 = \begin{pmatrix}1&-1&0\\0&0&0\\0&0&0\end{pmatrix}.
	\end{align*}
	\item Move the whole image by 1 to the bottom which gives 
	\begin{align*}
		v^1 = \begin{pmatrix}1&1&1\\1&1&1\\1&1&1\end{pmatrix},\quad v^2 = \begin{pmatrix}0&0&0\\0&0&0\\0&0&0\end{pmatrix}.
	\end{align*}
	\item Move the pixels at position $(1,1)$ and $(1,2)$ by 1 to the bottom and exchange some pixels in the third column 
	\begin{align*}
		v^1 = \begin{pmatrix}1&1&1\\0&0&1\\0&0&-2\end{pmatrix},\quad v^2 = \begin{pmatrix}0&0&0\\0&0&0\\0&0&0\end{pmatrix}.
	\end{align*}
\end{enumerate}
All these velocity fields produce the same result, that is they all fulfill the optical flow constraint $u_t+\nabla u\cdot \boldsymbol{v} = 0$. This results from the fact that for given images the optical flow constraint states for every point $x\in\Omega$ only one equation for a 2-dimensional velocity field $\boldsymbol{v}$. The solution is unique iff there exists a unique bijection between $u_1$ and $u_2$. This requires both images to consist of the same intensity values, which furthermore have to be unique. Unfortunately, this assumption is far from reality, since regions of constant intensity are characteristic for background or objects. Hence, for practical problems we have a highly underdetermined system and consequently a huge variety of possible underlying velocity fields $\boldsymbol{v}$.\\
As we will see later an error measure always explicitly or implicitly favors one of these possible results over the others. Those error measures, which explicitly prefer one result, expect a given ground truth field $\boldsymbol{v}_{GT}=(v^1_{GT},v^2_{GT})$ and measure a distance $d(\boldsymbol{v}_{GT},\boldsymbol{v})$ from the calculated velocity field to the given ground truth field. This strategy is questionable because in real world examples we usually do not have a ground truth velocity field, and artificial examples can usually be generated by several different ground truth fields. Hence, setting the ground truth $\boldsymbol{v}_{GT}$ directly favors a subjective result. This is at least from the mathematical viewpoint questionable.

\subsubsection{Absolute Endpoint Error}
\label{motionEstimationSubsectionAEE}
Despite the fact that explicit error measures might be problematic they are often used in the literature and we will also use them to evaluate our algorithms. In $\cite{baker2011database}$ two explicit error measures have been presented. The most intuitive one is the average endpoint error (aee), proposed in $\cite{otte1994optical}$, which is the vector-wise Euclidean norm of the difference vector $\boldsymbol{v}-\boldsymbol{v}_{GT}$. The difference is divided by $\left|\Omega\right|$ and we have
\begin{align*}
	aee := \frac{1}{\left|\Omega\right|}\int_{\Omega}\sqrt{(v^1(x)-v^1_{GT}(x))^2 + (v^2(x)-v^2_{GT}(x))^2} d x,
\end{align*}
or in a discrete formulation
\begin{align*}
	AEE := \frac{1}{nPx}\sum_{i=1}^{nPx}\sqrt{(v^1(i)-v^1_{GT}(i))^2 + (v^2(i)-v^2_{GT}(i))^2}.
\end{align*}
Here, $nPx$ denotes the number of pixels.

\subsubsection{Angular Error} 
\label{motionEstimationSubsectionAE}
A second measure states the angular error (ae) which goes back to the work of Fleet and Jepson $\cite{fleet1990computation}$ and a survey of Barron \textit{et al.} $\cite{barron1994performance}$. Here $\boldsymbol{v}$ and $\boldsymbol{v}_{GT}$ are projected into the 3-D space (to avoid division by zero) and normalized by
\begin{align*}
	\hat{\boldsymbol{v}} := \frac{(v^1,v^2,1)}{\sqrt{\left\|\boldsymbol{v}\right\|^2 + 1}}, \quad \quad \hat{\boldsymbol{v}}_{GT} := \frac{(v_{GT}^1,v^2_{GT},1)}{\sqrt{\left\|\boldsymbol{v}_{GT}\right\|^2 + 1}}.
\end{align*}
The error is then calculated measuring the angle between $\hat{\boldsymbol{v}}$ and $\hat{\boldsymbol{v}}_{GT}$ in the continuous setting as
\begin{align*}
	ae := \frac{1}{\left|\Omega\right|}\int_{\Omega} \arccos(\hat{\boldsymbol{v}}(x)\cdot\hat{\boldsymbol{v}}_{GT}(x))d x,
\end{align*}
and in a discrete setting as
\begin{align*}
	AE := \frac{1}{nPx}\sum_{i=1}^{nPx}  \arccos(\hat{\boldsymbol{v}}(i)\cdot\hat{\boldsymbol{v}}_{GT}(i)).
\end{align*}
Besides the fact that the $\textit{ae}$ is very popular, we want to emphasize a serious drawback. To illustrate this drawback we took a vector $\boldsymbol{v}_{GT}=(v^1,v^2)$ and denote the Euclidean norm on the x-axis in Figure $\ref{fig:angularErrorProblem}$. This vector has been disturbed by an absolute and relative error of $0.01$ to create $\boldsymbol{v}_{abs}$ and $\boldsymbol{v}_{rel}$:
\begin{align*}
	\boldsymbol{v}_{abs} = \boldsymbol{v}_{GT} -(0.99,0.99)^T\\
	\boldsymbol{v}_{rel} = \frac{\boldsymbol{v}_{GT}}{\left\| \boldsymbol{v}_{GT}\right\|_2 } \cdot 0.99.
\end{align*}
Afterwards we calculated the angular error between $\boldsymbol{v}$ and $\boldsymbol{v}_{abs}$ resp. $\boldsymbol{v}$ and $\boldsymbol{v}_{rel}$ and denoted this error on the $y-axis$ of Figure $\ref{fig:angularErrorProblem}$.
The first obvious property is that the error shrinks for larger velocities. This is critical since larger velocities are usually coupled with the object we are interested in. In the contrary absolute errors in small velocities (e.g. background) are over-penalized. 
To make this clear, let us consider a simple example: Take an algorithm producing a all-zero velocity field. This result perfectly recovers the background, but causes errors in the region of a moving object. Consider on the other hand an algorithm that perfectly recovers the movement in the background, but introduces slight errors in parts of the object. Due to the fact that errors in the background are over-penalized, the first result might be preferred although it yields a worse result.
\begin{figure}\centering
	\includegraphics[width=.6\textwidth,natwidth=500,natheight=439]{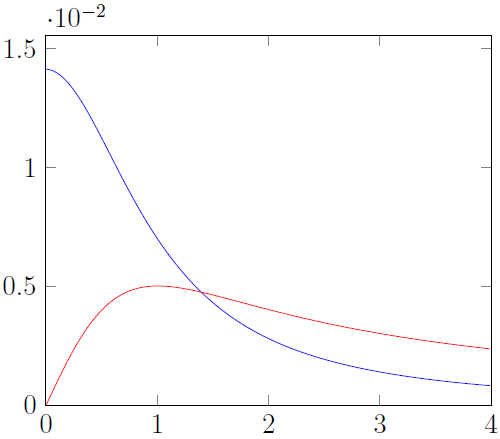}
	\caption[Plot to visualize drawbacks of the angular error]{Plot of the angular error \textit{ae} on the y-axis with increasing $\left\|\boldsymbol{v}\right\|$ on the x-axis. \\Blue: Absolute error $\left\| \boldsymbol{v}-\boldsymbol{v}_{GT}\right\|_2$ of 0.01\\Red: Relative error $\left\| \boldsymbol{v}-\boldsymbol{v}_{GT}\right\|_2 / \left\|\boldsymbol{v}_{GT}\right\|$ of 0.01}
	\label{fig:angularErrorProblem}
\end{figure}

\subsection{Evaluation}
In order to evaluate the introduced models we used the evaluation sequences from the Middlebury optical flow database \cite{middleburyURL}. First, the given ground-truth flow fields $\boldsymbol{v}_{GT}$ were scaled down to a magnitude of one pixel to overcome the limitations of our algorithms (see Section \ref{largeScaleOpticalFlow}). Afterwards, the evaluation images were created by keeping the first frame $I_1$ and generating the second frame $I_2$ by cubic interpolation of $I_1(x+\boldsymbol{v}_{GT})$. The noisy dataset was created with Gaussian noise ($\sigma=0.002$).\\
The evaluation was performed on the one hand with static parameters along all datasets and on the other hand by manually adapting the parameters to find smaller values in terms of AE and AEE. The static parameters can be found in Table \ref{parametersAndRanks}. Moreover, Table \ref{parametersAndRanks} contains the average rank of each algorithm for static and dynamic parameters over all evaluation datasets. Detailed error plots can be found in Figure \ref{fig:plotError}. An overview of typical results for the individual algorithms is printed in Figure \ref{figureResults2}.

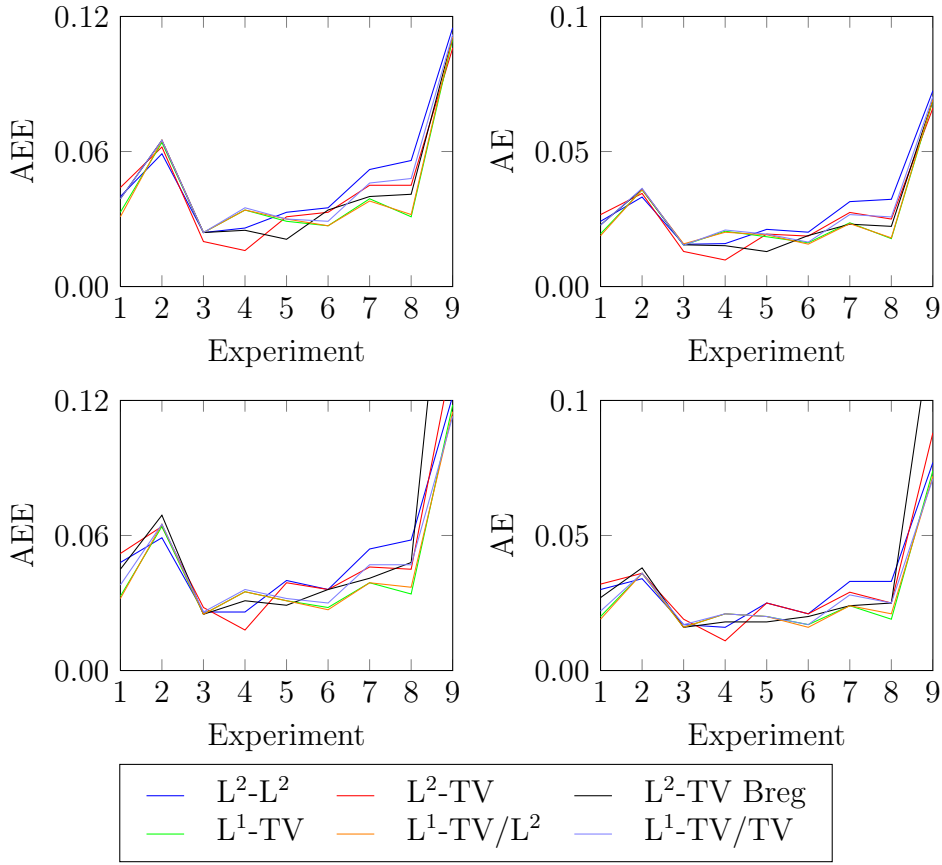
\begin{figure}
	\centering
	\begin{tikzpicture}[scale=1]
	\begin{axis}[
	xmin=1,
	xmax=9,
	ymin=0.00,
	ymax=0.12,
	width=6cm,
	xlabel=Experiment,
	xlabel near ticks,
	ylabel=AEE,
	ylabel near ticks,
	ytick={0.00,0.06,0.12},
	yticklabels={0.00,0.06,0.12},
	xtick={1,2,3,4,5,6,7,8,9},
	xticklabels={1,2,3,4,5,6,7,8,9},
	scaled ticks=false,
	tick label style={/pgf/number format/fixed} ]
	\addplot[mark=none,color=blue] table[x index = 0,y index = 1] from {evaluationResultBestAEE.csv};
	\addplot[mark=none,color=red] table[x index = 0,y index = 2] from {evaluationResultBestAEE.csv};
	\addplot[mark=none,color=green] table[x index = 0,y index = 3] from {evaluationResultBestAEE.csv};
	\addplot[mark=none,color=black] table[x index = 0,y index = 4] from {evaluationResultBestAEE.csv};
	\addplot[mark=none,color=orange] table[x index = 0,y index = 5] from {evaluationResultBestAEE.csv};
	\addplot[mark=none,color=blue!50] table[x index = 0,y index = 6] from {evaluationResultBestAEE.csv};
	\end{axis}
	\end{tikzpicture}
\begin{tikzpicture}[scale=1]
	\begin{axis}[
		xmin=1,
		xmax=9,
		ymin=0.00,
		ymax=0.1,
		width=6cm,
		xlabel=Experiment,
		xlabel near ticks,
		ylabel=AE,
		ylabel near ticks,
		ytick={0.00,0.05,0.1},
		yticklabels={0.00,0.05,0.1},
		xtick={1,2,3,4,5,6,7,8,9},
		xticklabels={1,2,3,4,5,6,7,8,9},
		scaled ticks=false,
		tick label style={/pgf/number format/fixed} ]
		\addplot[mark=none,color=blue] table[x index = 0,y index = 1] from {evaluationResultBestAE.csv};
		\addplot[mark=none,color=red] table[x index = 0,y index = 2] from {evaluationResultBestAE.csv};
		\addplot[mark=none,color=green] table[x index = 0,y index = 3] from {evaluationResultBestAE.csv};
		\addplot[mark=none,color=black] table[x index = 0,y index = 4] from {evaluationResultBestAE.csv};
		\addplot[mark=none,color=orange] table[x index = 0,y index = 5] from {evaluationResultBestAE.csv};
		\addplot[mark=none,color=blue!50] table[x index = 0,y index = 6] from {evaluationResultBestAE.csv};
	\end{axis}

\end{tikzpicture}
		
\begin{tikzpicture}
	\begin{axis}[
	xmin=1,
	xmax=9,
	ymin=0.00,
	ymax=0.12,
	width=6cm,
	xlabel=Experiment,
	xlabel near ticks,
	ylabel=AEE,
	ylabel near ticks,
	ytick={0.00,0.06,0.12},
	yticklabels={0.00,0.06,0.12},
	xtick={1,2,3,4,5,6,7,8,9},
	xticklabels={1,2,3,4,5,6,7,8,9},
	scaled ticks=false,
	tick label style={/pgf/number format/fixed} ]
	\addplot[mark=none,color=blue] table[x index = 0,y index = 1] from {evaluationResultFixedAEE.csv};
	\addplot[mark=none,color=red] table[x index = 0,y index = 2] from {evaluationResultFixedAEE.csv};
	\addplot[mark=none,color=green] table[x index = 0,y index = 3] from {evaluationResultFixedAEE.csv};
	\addplot[mark=none,color=black] table[x index = 0,y index = 4] from {evaluationResultFixedAEE.csv};
	\addplot[mark=none,color=orange] table[x index = 0,y index = 5] from {evaluationResultFixedAEE.csv};
	\addplot[mark=none,color=blue!50] table[x index = 0,y index = 6] from {evaluationResultFixedAEE.csv};
	\end{axis}
\end{tikzpicture}
\begin{tikzpicture}[scale=1]
	\begin{axis}[
		xmin=1,
		xmax=9,
		ymin=0.00,
		ymax=0.1,
		width=6cm,
		xlabel=Experiment,
		xlabel near ticks,
		ylabel=AE,
		ylabel near ticks,
		ytick={0.00,0.05,0.1},
		yticklabels={0.00,0.05,0.1},
		xtick={1,2,3,4,5,6,7,8,9},
		xticklabels={1,2,3,4,5,6,7,8,9},
		scaled ticks=false,
		tick label style={/pgf/number format/fixed} ]
		\addplot[mark=none,color=blue] table[x index = 0,y index = 1] from {evaluationResultFixedAE.csv};
		\addplot[mark=none,color=red] table[x index = 0,y index = 2] from {evaluationResultFixedAE.csv};
		\addplot[mark=none,color=green] table[x index = 0,y index = 3] from {evaluationResultFixedAE.csv};
		\addplot[mark=none,color=black] table[x index = 0,y index = 4] from {evaluationResultFixedAE.csv};
		\addplot[mark=none,color=orange] table[x index = 0,y index = 5] from {evaluationResultFixedAE.csv};
		\addplot[mark=none,color=blue!50] table[x index = 0,y index = 6] from {evaluationResultFixedAE.csv};
	\end{axis}
\end{tikzpicture}

\begin{tikzpicture}[scale=1]
	\node[draw=black, below left=2mm] {%
		\begin{tabular}{llllll}
			\raisebox{2pt}{\tikz{\draw[blue] (0,0) -- (5mm,0);}} &L$^2$-L$^2$ & \raisebox{2pt}{\tikz{\draw[red] (0,0) -- (5mm,0);}}& L$^2$-TV 		& \raisebox{2pt}{\tikz{\draw[black] (0,0) -- (5mm,0);}}  & L$^2$-TV Breg\\
			\raisebox{2pt}{\tikz{\draw[green] (0,0) -- (5mm,0);}}&L$^1$-TV & \raisebox{2pt}{\tikz{\draw[orange] (0,0) -- (5mm,0);}}& L$^1$-TV/L$^2$ & \raisebox{2pt}{\tikz{\draw[blue!50] (0,0) -- (5mm,0);}}& L$^1$-TV/TV
		\end{tabular}};
\end{tikzpicture}

	\caption{Plot of absolute endpoint error (AEE) and angular error (AE) for manually chosen parameter vaules (top) and fixed parameter vaules (bottom). The number on the x-axis denotes the following evaluation sequences from the middlebury database: 1. Dimetrodon, 2. Grove2, 3. Grove3, 4. Hydrangea, 5. Rubber Whale, 6. Urban2, 7. Urban3, 8. Venus.}
	\label{fig:plotError}
\end{figure}

\begin{figure}
	\centering
	\includegraphics[height=.14\textheight,natwidth=584,natheight=388]{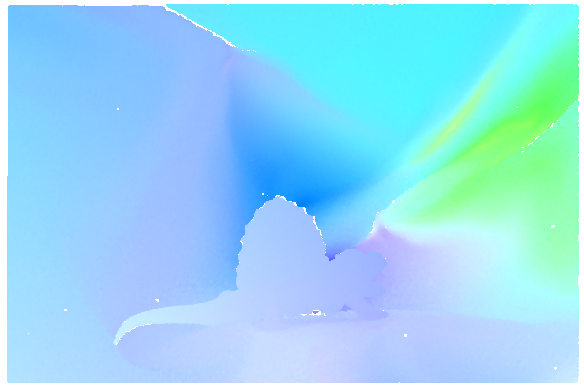}\enskip
	\includegraphics[height=.14\textheight,natwidth=584,natheight=388]{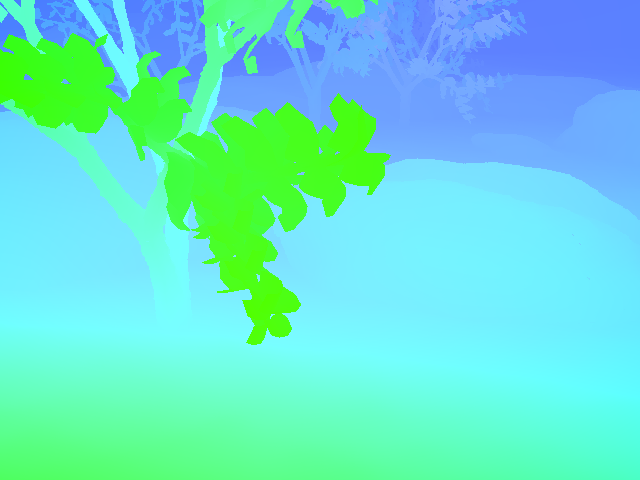}\enskip
	\includegraphics[height=.14\textheight,natwidth=584,natheight=388]{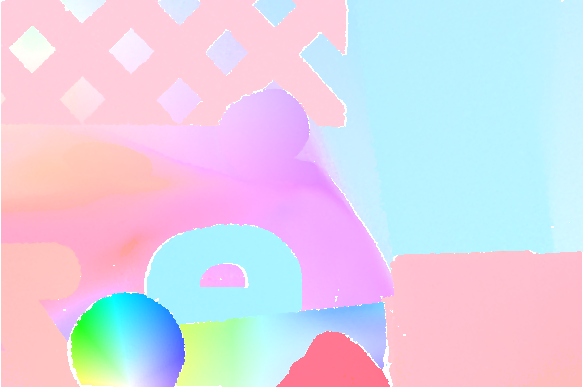}\\
	
	\includegraphics[height=.14\textheight,natwidth=584,natheight=388]{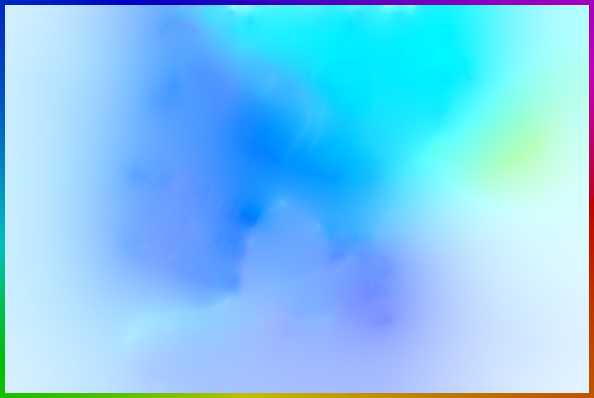}\enskip
	\includegraphics[height=.14\textheight,natwidth=584,natheight=388]{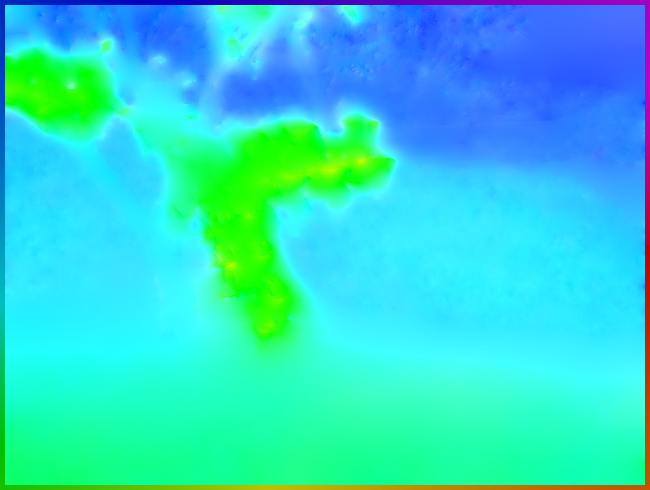}\enskip
	\includegraphics[height=.14\textheight,natwidth=584,natheight=388]{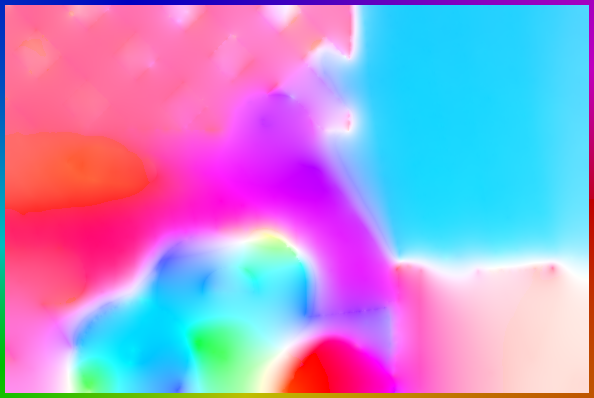}\\
	
	\includegraphics[height=.14\textheight,natwidth=584,natheight=388]{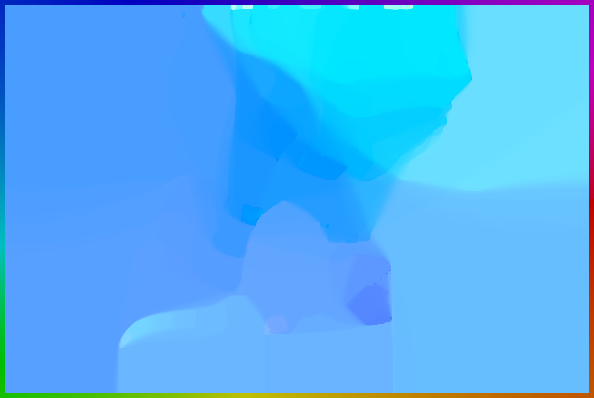}\enskip
	\includegraphics[height=.14\textheight,natwidth=584,natheight=388]{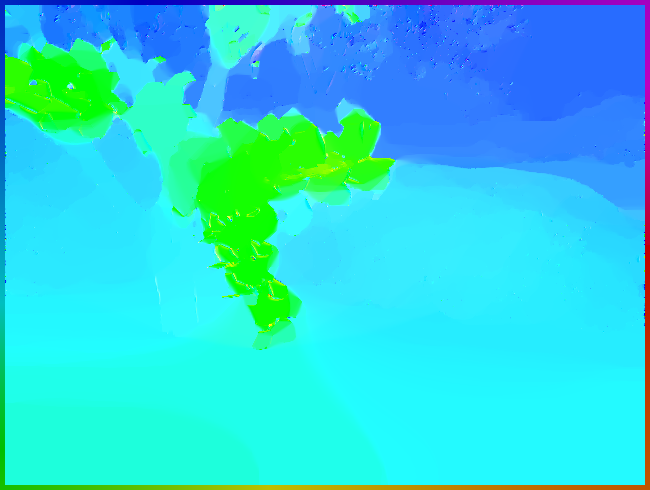}\enskip
	\includegraphics[height=.14\textheight,natwidth=584,natheight=388]{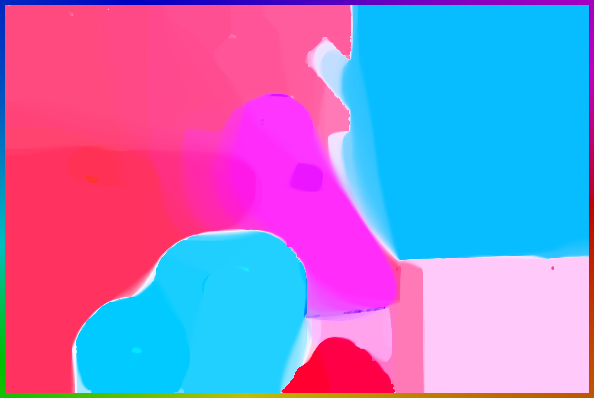}\\
	
	\includegraphics[height=.14\textheight,natwidth=584,natheight=388]{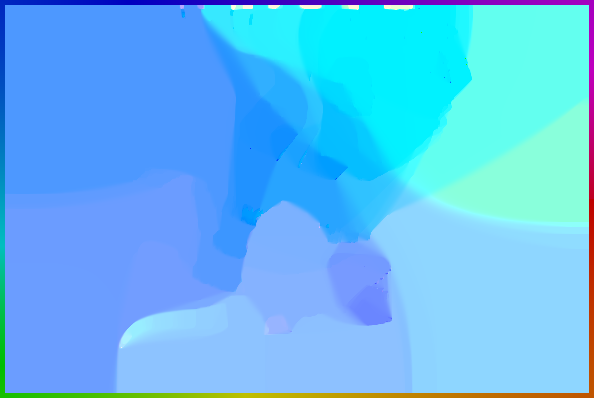}\enskip
	\includegraphics[height=.14\textheight,natwidth=584,natheight=388]{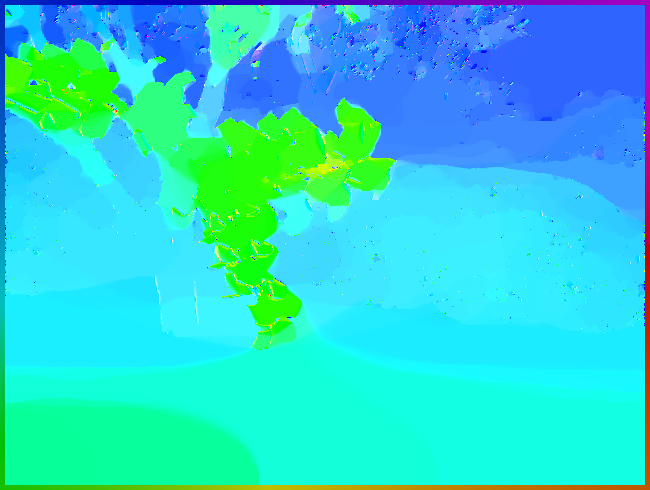}\enskip
	\includegraphics[height=.14\textheight,natwidth=584,natheight=388]{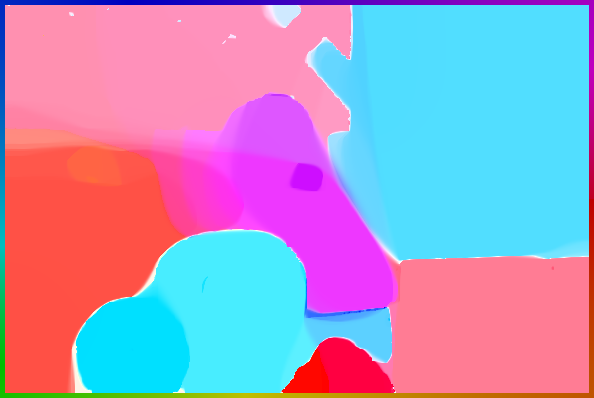}\\
	
	\includegraphics[height=.14\textheight,natwidth=584,natheight=388]{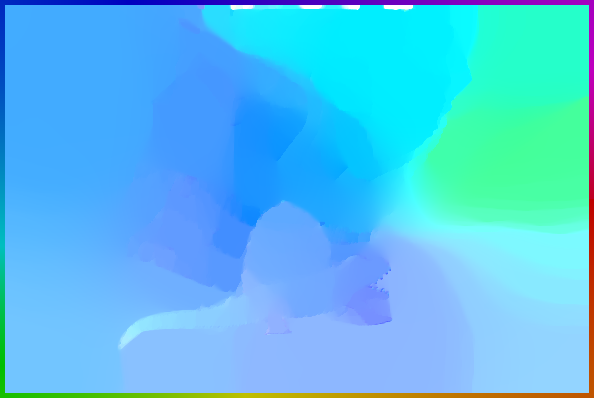}\enskip
	\includegraphics[height=.14\textheight,natwidth=584,natheight=388]{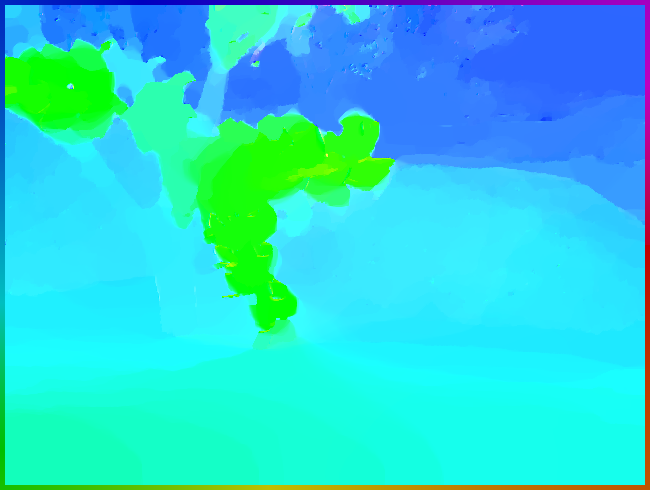}\enskip
	\includegraphics[height=.14\textheight,natwidth=584,natheight=388]{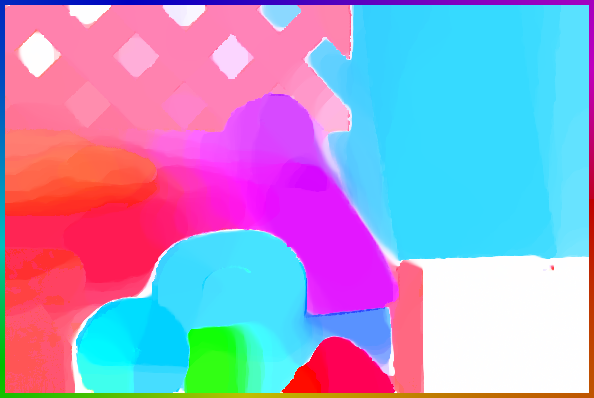}\\
	
	\includegraphics[height=.14\textheight,natwidth=584,natheight=388]{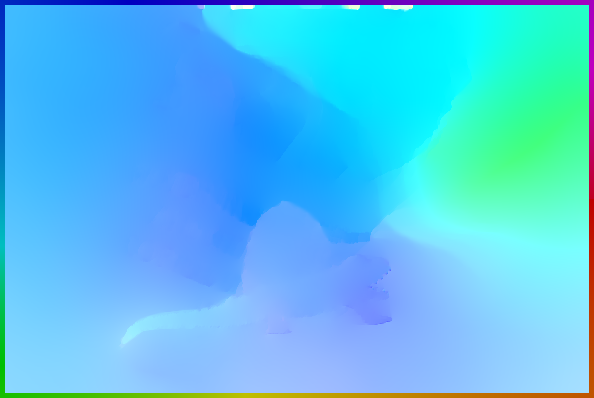}\enskip
	\includegraphics[height=.14\textheight,natwidth=584,natheight=388]{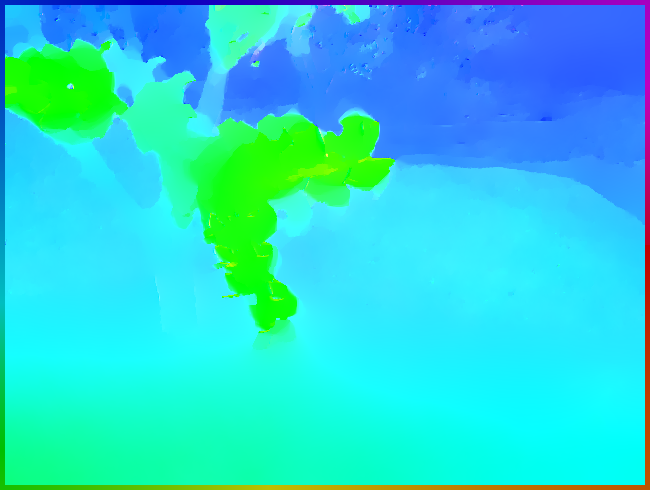}\enskip
	\includegraphics[height=.14\textheight,natwidth=584,natheight=388]{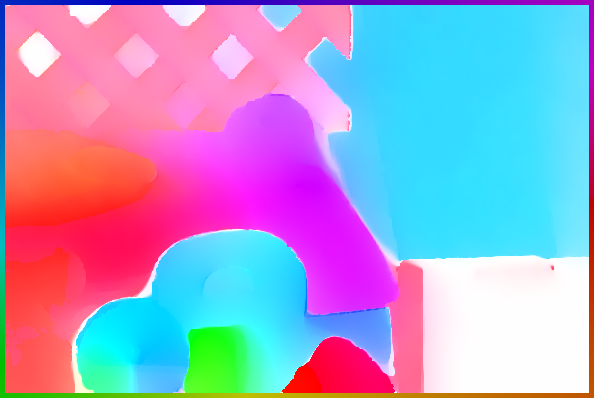}\\
	
	\includegraphics[height=.14\textheight,natwidth=584,natheight=388]{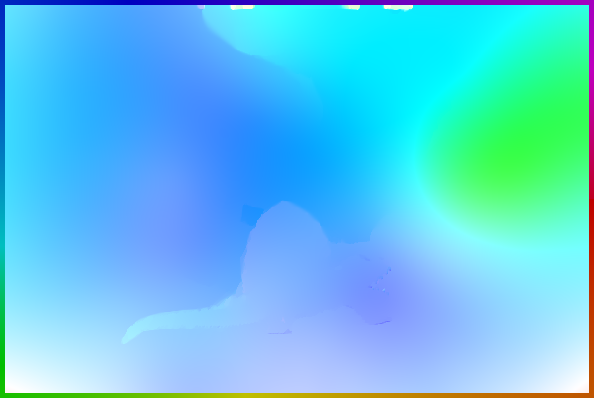}\enskip
	\includegraphics[height=.14\textheight,natwidth=584,natheight=388]{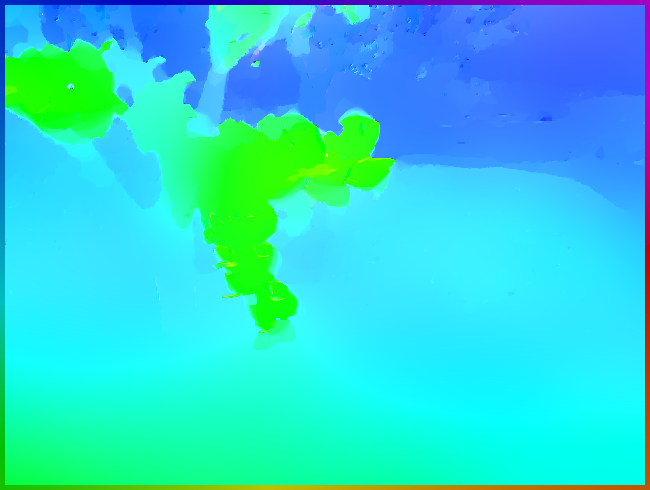}\enskip
	\includegraphics[height=.14\textheight,natwidth=584,natheight=388]{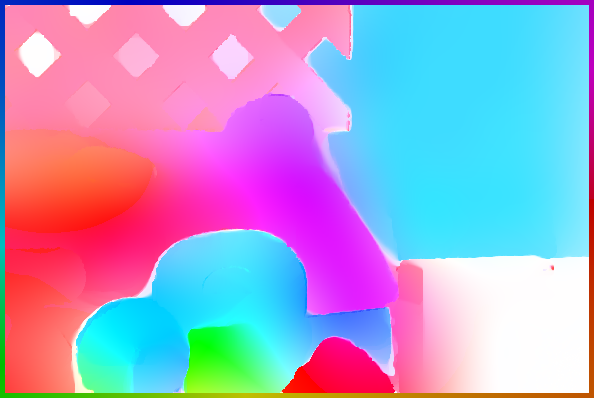}\\
	
	\caption{Overview of the results for Dimetrodon (left row), Grove2 (middle row) and Rubber Whale datasets (right row). Top to bottom: ground-truth, L$^2$-L$^2$, L$^2$-TV, L$^2$-TV Bregman, L$^1$-TV, L$^1$-TV/L$^2$, L$^1$-TV/TV}
	\label{figureResults2}
\end{figure}

\begin{table}
	\centering\footnotesize
	\begin{tabular}{|c|c|l|l|l|l|l|}
		\hline
		\multicolumn{3}{|c|}{} & \multicolumn{2}{c|}{Static parameters} & \multicolumn{2}{c|}{Dynamic parameters}\\
		\hline
		{Algorithm} & {Static $\alpha$} & {Static $\alpha_2$} & {$\varnothing AEE$} & {$\varnothing AE$}& {$\varnothing AEE$} & {$\varnothing AE$}\\
		\hline
		L$^2$-L$^2$ & 0.15 & - & 1.319 & 1.333 & 1.360 & 1.367\\
		\hline
		L$^2$-TV & 0.002 & - & 1.250 & 1.266 & 1.200 & 1.194\\
		\hline
		L$^2$-TV Breg & 0.02 & 10 & 1.345 & 1.286 & 1.200 & 1.171\\
		\hline
		L$^1$-TV & 0.1 & - &  1.135 & 1.137 & 1.213 & 1.212\\
		\hline
		L$^2$-TV/L$^2$ & 0.1 & 50 & 1.136 & 1.133 & 1.215 & 1.211\\
		\hline
		L$^2$-TV/TV & 0.1 & 1 & 1.237 & 1.205 & 1.341 & 1.315\\
		\hline
	\end{tabular}
	\caption{Table contains static parameters and averaged ranks in terms of AAE and AE. The average is calculated over AEE or AE of each result divided by the smallest AEE or AE of the particular dataset, i.e. $\varnothing AAE \mathrel{\mathop{\scriptsize:}} = \varnothing \left( \frac{AEE}{\min AEE} \right)$ and $\varnothing AE \mathrel{\mathop{\scriptsize:}} = \varnothing \left( \frac{AE}{\min \, AE} \right)$}
	\label{parametersAndRanks}
\end{table}

\begin{table}
	\centering\footnotesize
	\begin{tabular}{|c|c|l|l|l|l|l|}
		\hline
		{Error} & {Forward} & {Central} & {Interpolated}\\
		\hline
		AEE & 0.0515 & 0.0352 & 0.0221\\
		\hline
		AE & 0.0326 & 0.0207 & 0.0131\\
		\hline
	\end{tabular}
	\caption{Table contains evaluation for different approximations of the image gradient $\nabla u$ tested with the L$^1$-TV model on the Dimetrodon dataset with fixed parameter $\alpha=0.05$}
	\label{parametersAndRanks}
\end{table}

\section{Conclusion and Outlook}
The first result from the comparison of fixed versus variable parameters is that it only has small effect on the error and, consequently, the parameters do not have to be highly tuned for motion estimation. \\ 
Another important result from the evaluation is the effect of noise to the motion estimation process. Already a small level of Gaussian noise massively disturbs the motion estimation. In general, the models with L$^1$ data fidelity perform slightly better. Hence, images should be denoised beforehand or during the motion estimation process (see \cite{dirks}).\\
Taking into account the visual impression of Figure \ref{figureResults2} we notice that the grid structures in the background and the rotation on the lower left side are better reconstructed by the L$^1$ models. The L$^2$-L$^2$ model also seems to be able to detect the reconstruction, but due to the missing sharp edges the visual impression is poor. On the other hand, focusing on the constant moving block on the lower right side, the L$^2$-TV Bregman model reconstructs this part better than all the other models.

\subsection{Mass Preservation}
So far, the optical flow equation was used as data term in combination regularizers modeling different a-priori assumptions to the motion field $v$. Another possible approach to this problem is given by the continuity equation. Following the assumption that the total mass in the image data keeps constant in time, i.e. $\int_{\Omega}u(\cdot,t)dx=C\ \forall t$ and, furthermore, that mass is moved by a continuous flow $\boldsymbol{v}$ it can be shown that the flow satisfies the continuity equation $u_t+\nabla\cdot (vu)$. This equation, referring to the optical flow constraint often denoted as mass-preservation constraint, can be used as a data term in combination with each of the introduced regularizers.

\subsection{Higher Dimensions}
Nowadays, estimating the flow for higher dimensional image data is a current field of research. For example, microscopes record 3-dimensional data of moving cells and intracellular flows. Also in a medical context 3-dimensional motion estimation is of great importance, e.g. in modalities with canonical time resolution like PET and SPECT, or in MR imaging. It is e.g. possible to exactly map tumors of lung cancer patients despite of the respiratory motion based on (3d+t)-CT data. Unfortunately, the aperture problem becomes even worse in higher dimensions. Having in two spatial dimensions one equation (the optical flow constraint) for the two unknowns $v^1,v^2$ of the velocity field, we have to deal with three unknowns for 3 spatial dimensions. This missing information has to be caught by the regularizer in the variational model. Consequently, a-priori assumptions to the flow field have a even higher impact to the solution. For biological applications we would typically expect an incompressible flow (i.e. $\nabla\cdot v=0$), which can be used as an additional regularizer. For 3-dimensional computing of optical flow see e.g. \cite{barron2005tutorial}.

\subsection{Joint Models} 
Another interesting question in the context of motion estimation is robustness towards noise. It can be shown that already low levels of noise in the image data massively disturb the motion estimation process. Consequently, for many applications the image data gets denoised as a preprocessing step. An alternative approach is given by the joint motion estimation and image reconstruction models, which is even inevitable in the case of missing data, e.g. partly in space or time. For given image data $f$, we would typically minimize models of the following structure 
\begin{align*}
	\min_{u,v}\mathcal{D}(u,f) + \alpha\mathcal{R}(u) + \beta\mathcal{S}(v) + \gamma\mathcal{C}(u,v),
\end{align*}
where $\mathcal{D}(u,f)$ connects the unknown image data $u$ with $f$, $\mathcal{R}(u)$ (e.g. total variation) and $\mathcal{S}(v)$ (e.g. total variation or smoothness) act as regularizers for image data $u$ and velocity field $v$. The most important last part $\mathcal{C}(u,v)$ (optical flow or mass preservation constraint) connects image data and underlying motion field. One general benefit of such models is that the underlying motion is directly used for the image reconstruction, hence can be thought of as a motion compensation technique. On the other hand, the enhanced image data should generate more reasonable velocity fields (cf. \cite{dirks})

\subsection{Large Displacements}
\label{largeScaleOpticalFlow}
One main drawback of the previously introduced models is their limitation to flows of small magnitude. The reason for this lies in the optical flow constraint, where we assumed a small timestep $\delta_t$ in the modeling part. For large scale flows this assumption is simply not fulfilled anymore. Moreover, large errors in the discretization of the image gradients $u_x,u_y$ and $u_t$ are introduced when proceeding to longer timesteps.\\
A common approach to overcome these difficulties is to use a different linearization of the optical flow formulation $u_2(x+v)-u_1(x)=0$ in combination with a multiscale strategy. The brightness constancy assumption can be linearized using an approximation $v_0$ to $v$, yielding the following modified optical flow constraint: 
\begin{align*}
	\nabla u_2(x+v_0)\cdot (v-v_0) + u_2(x+v_0) - u_1(x) = 0.
\end{align*}
Acting as a data term in a variational model, approximations $v_0$ are obtained by performing the flow estimation on different downscaled versions of $u_1$ and $u_2$. Starting on a very coarse scale, an initial $v_0$ is calculated. This initial $v_0$ then acts as a approximation in the next finer scale. Following this so-called \textit{warping} strategy, we finally arrive at the original image. \\
It should be considered that it might be useful to adapt the regularization parameter in every warping step. Using the same regularization parameter independent of the size of the grid enhances the significance of a small error in an image, e.g. caused by noise, in a finer grid compared to its significance in a coarser grid. By reducing the influence of the regularization whilst refining the grid circumvents this mismatch.\\
A requirement for these multiscale approaches is that smaller structures move in a similar way to the larger ones, since the movements found in the coarse grids serve as an initial guess for the movements in finer grids. If small structures move completely independent of large ones, the flow field will not be detected correctly.\\
Recent approaches overcome these problems by using the idea of descriptor matching for optical flow. Descriptor matching is based on finding identical features in different images and matching them. This way it is possible to find extremely large displacements of discrete regions in between two images. However, since it is not possible to match every single pixel, this approach alone does not yield absolutely precise results. In \cite{brox2011large}, Brox and Malik combine descriptor matching with a variational approach to be able to estimate large displacements as well as precise results. Based on this work, Weinzaepfel at al. propose an algorithm for combined descriptor matching and variational motion estimation that also allows for non-rigid movements \cite{weinzaepfel2013deepflow}. In \cite{revaud2015epicflow}, Revaud et al. assume that edges in images match with edges in the flow field and incorporate an interpolation framework to the model of \cite{weinzaepfel2013deepflow} to avoid error-propagation caused by a warping. For evaluating large scale motion we may refer to the KITTI Vision Benchmark Suite \cite{Geiger2012CVPR,Menze2015CVPR}, which contains high-resolution real world datasets.

\bibliographystyle{plain}
\bibliography{citations}

\end{document}